\documentclass[12pt,oneside,openany,article]{memoir}
\usepackage{cmap}

\nouppercaseheads
\usepackage[amsthm,thmmarks]{ntheorem}
\usepackage[noTeX]{mmap}
\usepackage{etex}

\usepackage[russian, english]{babel}
\usepackage[leqno]{mathtools}

\usepackage{mathrsfs}
\usepackage{upgreek}
\usepackage[compress,square,comma,numbers]{natbib}
\usepackage{hypernat}

\usepackage{eso-pic}
\usepackage{sfmath}
\usepackage{tikz}
\usetikzlibrary{spy}

\usepackage{todonotes}

\usepackage{subfig}

\usepackage{microtype}

\usepackage{array}

\setsecheadstyle{\large\bfseries\memRTLraggedright}
\setsubsecheadstyle{\bfseries\memRTLraggedright}

\usepackage{graphicx,xcolor}
\definecolor{refkey}{rgb}{0,0,1}
\definecolor{labelkey}{rgb}{1,0,0}

\usepackage{hyperxmp}
\usepackage{xr-hyper}
\usepackage{nameref}
\usepackage[pdftex,bookmarks,pdfnewwindow,plainpages=false,unicode]{hyperref}

\usepackage{bookmark, url}

\usepackage{enumitem}
\usepackage{amssymb}

\newenvironment{claim}[1][{\textup{(\theequation)}}]{\refstepcounter{equation}\vglue10pt
\begin{trivlist}
\item[{\hskip\labelsep#1}]}{\vglue10pt\end{trivlist}}

\hypersetup{
colorlinks=true,
linkcolor=black,
citecolor=black,
urlcolor=blue,
pdfauthor={Victor Ivrii},
pdftitle={Bethe-Sommerfeld conjecture in semiclassical settings},
pdfsubject={Sharp Spectral Asymptotics},
pdfkeywords={Microlocal Analysis, sharp  spectral asymptotics, integrated density of states, periodic operators, Bethe-Sommerfeld conjecture},
bookmarksdepth={4}
}

\theoremstyle{plain}
\newtheorem{theorem}{Theorem}[chapter]

\newtheorem{proposition}[theorem]{Proposition}
\newtheorem{corollary}[theorem]{Corollary}

\theoremstyle{definition}

\theoremstyle{remark}
\newtheorem{remark}[theorem]{Remark}

\makeatletter
\newtheoremstyle{plainfoot}%
 {\item[\hskip\labelsep \theorem@headerfont ##1\ ##2\,\footnotemark\theorem@separator]}%
 {\item[\hskip\labelsep \theorem@headerfont ##1\ ##2\ (##3)\, \footnotemark\theorem@separator]}
\makeatother

\theoremstyle{plainfoot}
\theoremseparator{.}
\newtheorem{theorem-foot}[theorem]{Theorem}
\newtheorem{lemma-foot}[theorem]{Lemma}
\newtheorem{proposition-foot}[theorem]{Proposition}
\newtheorem{corollary-foot}[theorem]{Corollary}
\newtheorem{conjecture-foot}[theorem]{Conjecture}
\newtheorem{condition-foot}[condition]{Condition}

\theoremstyle{plainfoot}
\theoremseparator{.}
\theorembodyfont{}
\newtheorem{definition-foot}[theorem]{Definition}
\newtheorem{Problem-foot}[theorem]{Problem}

\theoremstyle{plainfoot}
\theoremseparator{.}
\theorembodyfont{}
\theoremheaderfont{\slshape}
\newtheorem{remark-foot}[theorem]{Remark}     
\newtheorem{example-foot}[theorem]{Example}
\newtheorem{problem-foot}[theorem]{Problem}

\numberwithin{equation}{chapter}

\newcounter{note}

\newcommand{\B}{\mathsf{B}}

\newcommand{\new}{\mathsf{new}}

\newcommand{\y}{\mathsf{y}}

\newcommand{\bR}{\mathbb{R}}

\newcommand{\bZ}{\mathbb{Z}}
\newcommand{\bS}{\mathbb{S}}

\newcommand{\cA}{\mathcal{A}}
\newcommand{\cB}{\mathcal{B}}

\newcommand{\cJ}{\mathcal{J}}

\newcommand{\cO}{\mathcal{O}}

\newcommand{\cT}{\mathcal{T}}

\newcommand{\cV}{\mathcal{V}}

\newcommand\sL{\mathscr{L}}
\newcommand\sH{\mathscr{H}}

\newcommand{\dist}{\operatorname{dist}}
\newcommand{\diam}{\operatorname{diam}}

\newcommand{\Spec}{\operatorname{Spec}}

\newcommand{\N}{\mathsf{N}}

\newcommand {\fH}{\mathfrak{H}}

\newcommand {\fV}{\mathfrak{V}}
\newcommand {\fW}{\mathfrak{W}}
\newcommand {\fX}{\mathfrak{X}}

\newcommand {\fZ}{\mathfrak{Z}}

\addtopsmarks{headings}{}{%
\createmark{chapter}{right}{shownumber}{}{. \ }
}
\title{Bethe-Sommerfeld conjecture in semiclassical settings 
}

\author{Victor Ivrii\thanks{This research was supported in part by National Science and Engineering Research Council (Canada) Discovery Grant RGPIN 13827}}

\begin{document}

\maketitle

\begin{abstract}
Under certain assumptions (including $d\ge 2)$ we prove that the spectrum of  a scalar operator in $\sL^2(\bR^d)$
\begin{equation*}
A_\varepsilon (x,hD)= A^0(hD) + \varepsilon B(x,hD),
\end{equation*}
covers interval $(\tau-\epsilon,\tau+\epsilon)$,
where $A^0$ is an elliptic operator and $B(x,hD)$ is a periodic perturbation, $\varepsilon=O(h^\varkappa)$, $\varkappa>0$.

Further, we consider generalizations.
\end{abstract}

\chapter{Introduction}
\label{sect-1}

\section{Preliminary remarks}
\label{sect-1.1}

 This work is inspired by a paper \cite{ParSob} by L.~Parnovski and A.~Sobolev, in which a classical Bethe-Sommerfeld conjecture was proven for operators $(-\Delta)^m +B(x,D)$ with operator $B$ of order $<2m$. In this  paper the crucial role was played by a (pseudodifferential) gauge transformation and thorough analysis of the resonant set, both introduced in the papers of L.~Parnovski and R.~Shterenberg \cite{ParSht1, ParSht2, ParSht3}, S.~Morozov, L.~Parnovski and R.~Shterenberg \cite{MorParSht} and earlier papers by A.~Sobolev \cite{Sob1, Sob2}, devoted to complete asymptotics of the integrated density of states. 

Later in \cite{IvrIDS} I used  the gauge same transformation in the  semiclassical settings, which allowed me to generalize the results and simplify the proofs of those papers\footnote{\label{foot-1} The other components of the proof were not only completely different, but in the framework of the different paradigm.}. Now  I would like to apply this gauge transform to Bethe-Sommerfeld conjecture in the semiclassical settings. The results obtained are more general (except the smoothness with respect to $\xi$ assumptions in \cite{ParSob} are more general than here) and the proofs are simpler.

Consider a scalar self-adjoint $h$-pseudo-differential operator $A_h\coloneqq A(x,hD)$ in $\bR^d$ with the Weyl symbol $A(x,\xi)$, such that\footnote{\label{foot-2} In fact, we consider $A_h\coloneqq A(x,hD,h)$.}
\begin{gather}
|D^\alpha _x D^\beta_\xi A(x,\xi)|\le c_{\alpha\beta}(|\xi|+1)^m \qquad \forall \alpha,\beta 
\label{eqn-1.1}\\
\shortintertext{and}
A(x,\xi)\ge c_0 |\xi|^m - C_0 \qquad 
\forall (x,\xi)\in \bR^{2d}.
\label{eqn-1.2}
\end{gather}

Then $A_h$ is semibounded from below. Also we assume that it is $\Gamma$-periodic with the \emph{lattice of periods} $\Gamma$:
\begin{equation}
A(x+\y ,\xi)= A(x,\xi) \qquad \forall x\in \bR^n\quad \forall \y \in \Gamma.
\label{eqn-1.3}
\end{equation}
We assume that $\Gamma $ is non-degenerate\footnote{\label{foot-3} I.e. with the $\Gamma=\{\y=n_1\y_1+\ldots+n_d\y_d,\ (n_1,\ldots,n_d)\in \bZ^d\}$ with linearly independent $\y_1,\ldots,\y_d\in \bR^d$.} and denote by $\Gamma^*$ the \emph{dual lattice}: 
\begin{equation}
\gamma \in \Gamma^* \iff \langle \gamma, \y\rangle \in 2\pi \bZ \quad \forall \y\in \Gamma;
\label{eqn-1.4}
\end{equation}
since we  use $\Gamma^*$ and it's elements in the paper much more often, than $\Gamma$ and it's elements, it is more convenient for us to reserve notation $\gamma$ for elements of $\Gamma^*$.

Also let $\cO  = \bR^d/\Gamma$ and $\cO ^* =\bR^d/\Gamma^*$ be \emph{fundamental domains}; we identify them with domains in $\bR^d$.

It is well-known that $\Spec (A)$ has a \emph{band-structure}. Namely, consider in $\sL^2(\cO )$ operator 
$A_h(\upxi)= A(x,hD)$ with the \emph{quasi-periodic boundary condition}:
\begin{equation}
u(x+ \y )=  e^{i \langle \y,\upxi\rangle} u(x)\qquad \forall x\in \cO  \quad  \forall \y \in \Gamma 
\label{eqn-1.5}
\end{equation}
with $\upxi \in \cO ^*$; it is called a \emph{quasimomentum}.  Then $\Spec (A_h(\upxi))$ is discrete 
\begin{equation}
\Spec (A_h(\upxi))=\bigcup _n \lambda _{n,h} (\upxi)
\label{eqn-1.6}
\end{equation}
and depends on $\upxi$ continuously. Further,
\begin{equation}
\Spec (A_h)= \bigcup_{\upxi \in \cO ^*}\Spec (A_h(\upxi))\eqqcolon \bigcup_n \Lambda_{n,h},
\label{eqn-1.7}
\end{equation}
with the \emph{spectral bands}
$\Lambda_{n,h}\coloneqq\bigcup_{\upxi \in \cO ^*}\{\lambda _{n,h} (\upxi)\}$. 

One can prove that the with of the spectral band  near energy level $\tau$ is $O(h)$. Spectral bands could overlap but they also could leave uncovered intervals, called \emph{spectral gaps}. It follows from \cite{IvrIDS} that in our assumptions (see below) the width of the spectral gaps near energy level $\tau$ is $O(h^\infty)$. Bethe-Sommerfeld conjecture in the semiclassical settings claims that there are no spectral gaps near energy level $\tau$ (in the corresponding assumptions, which include $d\ge 2$).

\section{Main theorem (statement)}
\label{sect-1.2}
We assume that    
\begin{equation}
A_h\coloneqq A(x,hD)=A^0(hD)+\varepsilon B(x,hD),
\label{eqn-1.8}
\end{equation}
where $A^0(\xi)$ satisfies (\ref{eqn-1.1}), (\ref{eqn-1.2}) and $B(x,\xi)$ satisfies (\ref{eqn-1.1}) and (\ref{eqn-1.3}) and $\varepsilon >0$ is a small parameter.  For $A^0(\xi)$ instead of $\lambda_n(\upxi)$ we have 
\begin{equation}
\lambda^0_\gamma (\upxi)\coloneqq A^0(h(\gamma +\upxi))  \qquad\text{with\ \ } \gamma \in \Gamma^*.
\label{eqn-1.9}
\end{equation}
Recall that (as in \cite{IvrIDS})
\begin{gather}
B(x,\xi)=\sum_{\gamma \in \Gamma} b_{\gamma}(\xi)e^{i\langle \gamma,x\rangle}
\label{eqn-1.10}\\
\intertext{with $\Theta=\Gamma^*$ where due to (\ref{eqn-1.1})}
|D^\beta _\xi b_\gamma (\xi)|\le C_{L\beta} (|\gamma|+1)^{-L}(|\xi|+1)^m\qquad \forall \beta \quad 
\forall (x,\xi)\in \bR^{2d}
\label{eqn-1.11}
\end{gather}
with an arbitrarily large exponent $L$.

\begin{theorem}\label{thm-1.1}
Let $d\ge 2$ and let operator $A_h$ be given by \textup{(\ref{eqn-1.8})} with $\varepsilon=O(h^\vartheta)$ with arbitrary $\varkappa>0$ and with $A^0_h=A^0(hD)$ satisfying \textup{(\ref{eqn-1.1})}, \textup{(\ref{eqn-1.2})} and $B(x,\xi)$ satisfying \textup{(\ref{eqn-1.1})} and \textup{(\ref{eqn-1.3})}.

Further,  assume that the \emph{microhyperbolicity}  and \emph{strong convexity} conditions  on the energy level $\tau$ are fulfilled:
\begin{equation}
|A^0(\xi) - \tau| +|\nabla_\xi A^0(\xi)|\ge \epsilon_0
\label{eqn-1.12}
\end{equation}
and 
\begin{multline}
\pm \sum_{j,k} A^0_{\xi_j\xi_k}(\xi)\eta_j\eta_k \ge \epsilon_0 |\eta|^2\\
 \forall \xi\colon A^0(\xi)=\tau\ \ \forall \eta\colon \sum_j  A^0_{\xi_j}(\xi)\eta_j=0.
\label{eqn-1.13}
\end{multline}
Furthermore, assume that there exists $\xi\in \Sigma_\tau$ such that for every $\eta\in \Sigma_\tau$, $\eta\ne \xi$, such that $\nabla _\eta A^0(\eta)$ is parallel  to $\nabla _\xi A^0(\xi)$\,\footnote{\label{foot-4} I.e. $\eta A^0(\eta)=\nu \nabla _\xi A^0(\xi)$ with $\nu\ne 0$; we call $\eta$ \emph{antipodal pont}.}
\begin{claim}\label{eqn-1.14}
$\Sigma_\tau $, intersected with some vicinity of $\eta$ and shifted by $(\xi-\eta)$, coincides in the vicinity of $\xi$ with $\{\zeta\colon \zeta_k=g(\zeta_{\hat{k}}\}$ and
 $\Sigma_\tau$ coincides in the vicinity of $\xi$ with $\{\zeta\colon \zeta_k=f(\zeta_{\hat{k}}\}$\,\footnote{\label{foot-5} With $\zeta_{\hat{k}}$ meaning all coordinates except $\zeta_k$.} 
and  $\nabla^\alpha (f-g)(0)\ne 0$ for some $\alpha\colon|\alpha|=2$\,\footnote{\label{foot-6} Obviously $\nabla (f-g)(0)=0$. One can prove easily, that if this condition holds at $\xi$ with some $\alpha\colon |\alpha|>2$, then changing slightly $\xi$, we make it fulfilled with $|\alpha|=2$.}. 
\end{claim}

Then $\Spec (A_h) \supset [\tau-\epsilon,\tau+\epsilon]$ for sufficiently small $\epsilon>0$.
\end{theorem}
\enlargethispage{3\baselineskip}

\begin{remark}\label{rem-1.2}
\begin{enumerate}[label=(\roman*), wide, labelindent=0pt]
\item\label{rem-1.2-i}
If $\Sigma_\tau$ is strongly convex and connected then for every $\xi\in \Sigma_\tau$ there exists exactly one antipodal point $\eta\in \Sigma_\tau$; then $\nu<0$ and assumption (\ref{eqn-1.14}) is fulfilled. In particular, if $A^0(\xi)=|\xi|^{m}$,  then $\eta = -\xi$ and $\nu=-1$.

\item\label{rem-1.2-ii}
If $\Sigma_\tau$ is is strongly convex and consists of $p$ connected components, then the set  
$\fZ(\xi) =\{\eta\in \Sigma_\tau, \eta \ne \xi \colon \nabla _\eta A^0(\eta)\parallel  \nabla _\xi A^0(\xi)\}$ contains exactly $2p-1$ elements, and for $p$ of them $\nu<0$ and assumption (\ref{eqn-1.14}) is fulfilled for sure, while for $(p-1)$ of them  $\nu >0$. \end{enumerate}
\end{remark}

\section{Idea of the proof and the plan of the paper}
\label{sect-1.3}

One needs to understand, how gaps could appear, why they appear if $d=1$ and why it is not the case if $d\ge 2$. Observe that  $\lambda_n (\upxi)$ can be identified with some $\lambda^0_\gamma(\upxi)$ only locally, if 
$\lambda^0_\gamma(\upxi)$ is sufficiently different from $\lambda^0_{\gamma'}(\upxi)$ for any 
$\gamma'\ne \gamma$. 

Indeed, in the basis of eigenfunctions of $A^0_\upxi (hD)$\,\footnote{\label{foot-7} Consisting of $\exp (i\langle x, \gamma +\upxi\rangle)$.} perturbation $\varepsilon B(x, hD)$ can contain out-of-diagonal elements 
$\varepsilon b_{\gamma-\gamma'}(\upxi)$ and such identification is possible only if  
$|\lambda^0_\gamma(\upxi)-\lambda^0_{\gamma'}(\upxi)|$ is larger than the size of such element. 

If $d=1$, $A^0(\xi)=\xi^2$ and $\varepsilon \le \epsilon' h$ with sufficiently small $\epsilon'>0$ and 
$\tau \asymp 1$, it can happen only if $\gamma'$ coincides with $-\gamma$ or with one of two adjacent points in $\Gamma^*$ and  $|\upxi-\frac{1}{2}(\gamma +\gamma')|=O(\varepsilon h^\infty)$. This exclude from possible values of either  $\lambda^0_\gamma (\upxi)$ or $\lambda^0_{\gamma'}(\upxi)$ the interval of the width $O(\varepsilon h^\infty)$ and on such interval can happen (and really happens for a generic perturbation) the realignment:

\begin{figure}[h]
\centering
\subfloat[]{%
\begin{tikzpicture}
\draw[thin,->] (-1.5,-1.5)--(1.5,-1.5) node [below] {$\upxi$};
\draw (-1.5,-1) -- (1.5,1) node[right] {$\lambda^0_n(\upxi)$};
\draw (-1.5,1) -- (1.5,-1) node[right] {$\lambda^0_m(\upxi)$};
\end{tikzpicture}}
\qquad\qquad
\subfloat[]{%
\begin{tikzpicture}[line cap=round,line join=round,x=1cm,y=1cm,
     spy using outlines={circle,lens={scale=3}, size= 1cm, connect spies}]
    \spy [blue] on (0.0,0.0)   in node [right] at (2.5,0);
\draw[thin,->] (-1.5,-1.5)--(1.5,-1.5) node [below] {$\upxi$};
\draw (-1.5,-1) -- (-.15,-.1) .. controls (0,-.05) .. (.15,-.1)--(1.5,-1) node[right] {$\lambda_n(\upxi)$};
\draw (-1.5,1) -- (-.15,.1) ..   controls (0,.05) ..  (.15,.1)-- (1.5,1) node[right] {$\lambda_m(\upxi)$};
\draw[red] (0,-.055)--(0,.055);
\end{tikzpicture}}
\caption{Spectral gap is a red interval}
\label{Fig-1}
\end{figure}
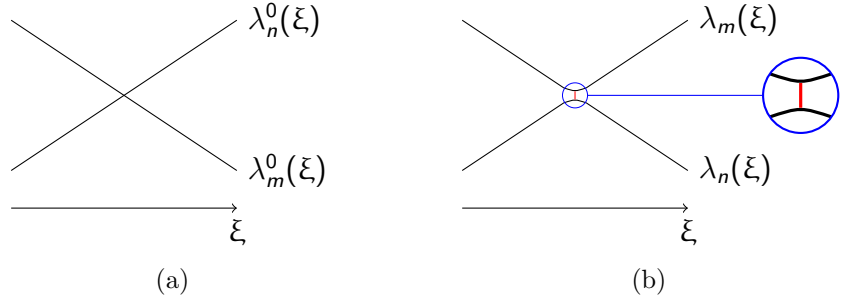
 
If $d\ge 2$ the picture becomes more complicated:  there are much more opportunities for 
$\lambda^0_\gamma (\upxi)$ and $\lambda^0_{\gamma'}(\upxi)$ to become close, even if $\gamma $ and $\gamma'$ are not that far away; on the other hand, there is a much more opportunities for us to select $\xi =h(\gamma +\upxi)\in \Sigma_\tau$ and then to adjust $\upxi$ so that $\xi =h(\gamma +\upxi)$ remains on $\Sigma_\tau$ but $\eta=h(\gamma' +\upxi)$ moves away from $\Sigma_\tau$ sufficiently far away\footnote{\label{foot-8} This will happen if either $\nabla_\eta A^0(\eta)$ differs from $\nu\nabla_\xi A^0(\xi)$, or if coincides with it but (\ref{eqn-1.14}) is fulfilled.} and then tune-up $\upxi$ once again so that $\tau\in \Spec(A_h(\upxi))$.

In fact, we prove the following statement which together with Theorem~\ref{thm-1.1} (which follows from it trivially) are semiclassical analogue of Theorem~2.1 of \cite{ParSob}:

\begin{theorem}\label{thm-1.3}
In the framework of Theorem~\ref{thm-1.1} there exist $n$ and $\upxi^*$ such that $\lambda_n(\upxi^*)=\tau$ and $\lambda_n (\upxi)$ covers interval  $[\tau-\upsilon h,\tau + \upsilon  h]$ when $\upxi$ runs ball 
$\B (\upxi^*, \upsilon)$ while  $|\lambda_m(\upxi)-\tau|\ge \epsilon \upsilon h$ for all $m\ne n$ and 
$\upxi\in\B (\upxi^*, \upsilon)$. Here 
\begin{equation}
\upsilon=\epsilon \left\{\begin{aligned}
h^{(d-1)^2} &\min(1 ,\,\varepsilon^{-3(d-1)/2}h^{(d-1)+\sigma}) && d\ge 3,\\
h  &\min(|\log h|^{-1} ,\,\varepsilon^{-3/2}h^{\sigma}) &&d=2
\end{aligned}\right.
\label{eqn-1.15}
\end{equation}
with arbitrarily small exponent $\sigma>0$.
\end{theorem}

Proof of Theorem~\ref{thm-1.3} occupies two next sections. In Section~\ref{sect-2} we reduce operator in the vicinity of $\Sigma_\tau$ to the block-diagonal form and study its structure. To do this we need to examine the structure of the resonant set of the operator. In Section~\ref{sect-3} we prove Theorem~\ref{thm-1.3} and thus Theorem~\ref{thm-1.1}.

Finally, in  Section~\ref{sect-4} we discuss our results and the possible improvements.

\chapter{Reduction of operator}
\label{sect-2}

\section{Reduction}
\label{sect-2.1}

On this step we reduce $A$ to the block-diagonal form in the vicinity~of~$\Sigma_\tau$
\begin{equation}
\Omega_\tau\coloneqq \{ \xi\colon |A^0(\xi) - \tau|\le C\varepsilon h^{-\delta}\}.
\label{eqn-2.1}
\end{equation}
In what follows, we assume that $\varepsilon \ge h$, i.e. 
\begin{equation}
h\le \varepsilon \le h^\vartheta,\qquad \vartheta >0.
\label{eqn-2.2}
\end{equation}

To do this we need just to repeat with the obvious modifications definitions and arguments of Sections~\ref{IDS-sect-1} and~\ref{IDS-sect-2}  of \cite{IvrIDS}. Namely, now  $\Theta\coloneqq \Gamma^*$ is a non-degenerate lattice rather than the pseudo-lattice, as it was in that paper, and all conditions \ref{IDS-cond-A}, \ref{IDS-cond-B}, \ref{IDS-cond-C}, \ref{IDS-cond-D}, and \ref{IDS-cond-E},
are fulfilled with $\Theta'\coloneqq\Theta \cap \B(0,\omega)$  with $\omega=h^{-\varkappa}$ where we select sufficiently small $\varkappa>0$  later and $\Theta'_K=\Theta \cap \B(0,K\omega)$ be an arithmetic sum of $K$ copies of $\Theta'$ with sufficiently large $K$ to be chosen later. 

We call point $\xi$ \emph{non-resonant} if 
\begin{equation}
|\langle \nabla _\xi A^0(\xi),\theta \rangle| \ge \rho \qquad \forall \theta \in \Theta'_K\setminus 0
\label{eqn-2.3}
\end{equation}
with $\rho\in [\varepsilon^{1/2}h^{-\delta},h^\delta]$ with arbitrarily small $\delta>0$. Otherwise we call it \emph{resonant}. More precisely
\begin{equation}
\Lambda  \coloneqq \bigcup_{\theta \in \Theta'_K\setminus 0} \Lambda (\theta),
\label{eqn-2.4}
\end{equation}
where $\Lambda (\theta)$ is the set of $\xi $, violating (\ref{eqn-2.3})  for given $\theta \in \Theta'_K\setminus 0$.

It obviously follows from the microhyperbolicity and strong convexity assumptions (\ref{eqn-1.12}) and (\ref{eqn-1.13}) that

\begin{claim}\label{eqn-2.5}
$\upmu_\tau $-measure\footnote{\label{foot-9} $\upmu_\tau=d\xi:dA^0(\xi)$ is a natural measure on $\Sigma_\tau$.} of  
$\Lambda \cap\Sigma_\tau$, does not exceed $C_0 r^{d-1}\rho$ and Euclidean measure of
$\Lambda\cap \{ \xi\colon |A^0(\xi) - \tau|\le \varsigma\}$ does not exceed $C_0 r^{d-1}\rho\varsigma$, where $r=K h^{-\varkappa}$.
\end{claim}

Indeed, $(d-1)$-dimensional measure of $\{x\colon |x|=1, \ |\langle x,\theta\rangle|\le \rho \}$ does not exceed $C_0|\theta|^{-1}\rho $ and
and due to microhyperbolicity and strong convexity assumptions maps $\Sigma_\tau \to \nabla A^0(\xi)/|\nabla A^0(\xi)|\in \bS^{d-1}$ and
\begin{equation*}
\{ \xi\colon |A^0(\xi) - \tau|\le \varsigma\}\to (\nabla A^0(\xi)/|\nabla A^0(\xi)|, A^0(\xi))\in \bR^d
\end{equation*}
 are diffeomorphisms.

Furthermore, according to Proposition \ref{IDS-prop-2.7} of \cite{IvrIDS} that on the non-resonant set 
one can ``almost'' diagonalize $A(x,hD)$. More precisely  

\begin{proposition}\label{prop-2.1}
Let assumptions \textup{(\ref{eqn-1.12})} and \textup{(\ref{eqn-1.13})} be fulfilled.
\begin{enumerate}[label=(\roman*), wide, labelindent=0pt]
\item\label{prop-2.1-i}
Then there exists a periodic pseudodifferential operator $P=P(x,hD)$ such that
\begin{gather}
\bigl(e^{-i\varepsilon h^{-1}P} A e^{i\varepsilon h^{-1}P} -\cA\bigr) Q \equiv 0
\label{eqn-2.6} \\
\shortintertext{with}
\cA = A^0(hD) + \varepsilon B'(hD) + \varepsilon B'' (x,hD)
\label{eqn-2.7}
\end{gather}
modulo operator from $\sH^m$ to $\sL^2$ with the operator norm $O(h^{M})$ with $M$ arbitrarily large and 
$K=K(M,d,\delta)$ in the definition of non-resonant point provided $Q=Q(hD)$ has a symbol, supported in 
$\{\xi\colon |A^0(\xi)-\tau|\le 2\varepsilon h^{-\delta}\}$.

Here $P(x,hD)$, $B'(hD)$ and $B'' (x,hD)$ are operator with Weyl symbols 
of the same form \textup{(\ref{eqn-1.10})} albeit such that
\begin{gather}
|D_\xi ^\alpha D_x^\beta P |\le c_{\alpha\beta} \rho^{-1-|\alpha|}
\qquad \forall \alpha, \beta,\label{eqn-2.8}\\
|D_\xi ^\alpha D_x^\beta B'' |\le c'_{\alpha\beta} \rho^{-|\alpha|}
\qquad \forall \alpha, \beta,
\label{eqn-2.9}
\end{gather}
and symbol of $B'$ also satisfies \textup{(\ref{eqn-2.9})}.

\item\label{prop-2.1-ii}
Further,
\begin{equation}
\xi\notin \Lambda(\theta) \implies b''_\theta (\xi)=0.
\label{eqn-2.10}
\end{equation}
and $B'(\xi)$ coincides with $b_0(\xi)$ modulo  $O(\varepsilon \rho^{-2})$.
\end{enumerate}
\end{proposition}

In what follows
\begin{equation}
\cA^0 (hD)\coloneqq A^0(hD)+ \varepsilon B'(hD)\qquad\text{and}\qquad  \cB\coloneqq  B''(x,hD).
\label{eqn-2.11}
\end{equation}

\begin{remark}\label{rem-2.2}
\begin{enumerate}[label=(\roman*), wide, labelindent=0pt]
\item\label{rem-2.2-i}
Eigenvalues of $\cA^0$ are
\begin{equation}
\lambda_\gamma (\upxi)=\cA^0 (h(\gamma+\upxi)).
\label{eqn-2.12}
\end{equation}

\item\label{rem-2.2-ii}
If $\xi = h(\gamma+\upxi)$ is non-resonant, then due to (\ref{eqn-2.10}) in $\epsilon'\rho$-vicinity of $\xi$ this $\lambda_\gamma (\upxi)$ is also an eigenvalue of $\cA$.

\item\label{rem-2.2-iii}
Without any loss of the generality one can assume that 
\begin{equation}
|\theta| \ge \varepsilon h^{-\delta}\implies b''_\theta=0.
\label{eqn-2.13}
\end{equation}
We assume that this is the case.

\item\label{rem-2.2-iv}
Let us replace operator $\cA$ defined by (\ref{eqn-2.7}) by operator
\begin{equation}
\cA' = \cA^0(hD) +  \varepsilon \cB'(x,hD),\qquad \cB'(x,hD)=T (hD)\cB T(hD)
\label{eqn-2.14}
\end{equation}
with $T(hD)$ operator with symbol $T(\xi)$ which is a characteristic function of 
$\Omega _\tau$ defined by (\ref{eqn-2.1}) with $C=6$. Then (\ref{eqn-2.6}) holds.

From now on $\cA\coloneqq \cA'$ and $\cB\coloneqq \cB'$.
\end{enumerate}
\end{remark}

It would be sufficient to prove Theorem~\ref{thm-1.3} for operator $\cA$. Indeed,

\begin{proposition}\label{prop-2.3}
\begin{enumerate}[label=(\roman*), wide, labelindent=0pt]
\item\label{prop-2.3-i}
For each point $\lambda \in \Spec (A(\upxi))\cap \{|\lambda-\tau|\le  \varepsilon h^{-\delta}\}$ \ 
 $\dist (\lambda , \Spec (\cA (\upxi))\le Ch^M$.
 
\item\label{prop-2.3-ii}
Conversely, for each point $\lambda \in \Spec (\cA (\upxi))\cap \{|\lambda-\tau|\le  \varepsilon h^{-\delta}\}$ \\ 
 $\dist (\lambda , \Spec (A(\upxi))\le Ch^M$. 

\item\label{prop-2.3-iii}
Furthermore, if $\lambda \in \Spec (A(\upxi))\cap \{|\lambda-\tau|\le  \varepsilon h^{-\delta}\}$ is
a simple eigenvalue separated from the rest of  $\Spec (A(\upxi))$ by a distance at least $2h^{M-1}$, then
there exists   $\lambda '\in \Spec (\cA (\upxi))\cap \{|\lambda'-\lambda |\le  Ch^M\}$
separated from the rest of $\Spec (\cA (\upxi))$ by a distance at least $h^{M-1}$.

\item\label{prop-2.3-iv}
Conversely, if $\lambda '\in \Spec (\cA (\upxi))\cap \{|\lambda-\tau|\le  \varepsilon h^{-\delta}\}$ is
a simple eigenvalue separated from the rest of $\Spec (\cA (\upxi))$  by a distance at least $2h^{M-1}$, then
there exists   $\lambda \in \Spec (A(\upxi))\cap \{|\lambda'-\lambda |\le  Ch^M\}$
separated from the rest of $\Spec (A(\upxi))$  by a distance at least $h^{M-1}$.
\end{enumerate}
\end{proposition}

\begin{proof}
Trivial proof is left to the reader.
\end{proof}

\begin{remark}\label{rem-2.4}
One can generalize Statements \ref{prop-2.3-iii} and \ref{prop-2.3-iv} of Proposition~\ref{prop-2.3} from simple eigenvalues to subsets 
of $\Spec (A(\upxi))$ and $\Spec (\cA (\upxi))$ separated by the rest of the spectra; these subsets will contain the same number of elements.
\end{remark}

\section{Classification of resonant points}
\label{sect-2.2}

We start from the case $d=2$. Then We have only one kind of resonant points $\Xi_1=\Lambda$. If $d\ge 3$ then there are $d-2$ kinds of resonant points. First, following \cite{ParSob} consider \emph{lattice spaces} $\fV$ spanned by $n$ linearly independent elements 
$\theta_1,\ldots,\theta_n\in \Gamma^*\cap \B(0,r)$, where we take $r=K h^{-\varkappa}$. Let $\cV_n$ be the set of all such spaces.

It is known \cite{ParSob} that 

\begin{proposition}\label{prop-2.5}
For $\fV\in \cV_n$, $\fW\in \cV_m$ either $\fV\subset\fW$, or $\fW\subset\fV$ or the angle\footnote{\label{foot-10} This angle $\widehat{(\fV,\fW)}$ is defined as the smallest possible angle between two vectors $v\in \fV\ominus (\fV\cap\fW)$ and  $w\in \fW\ominus (\fV\cap\fW)$.}  between $\fV$ and $\fW$ is at least $\epsilon r^{-L}$ with $L=L(d)$ and $\epsilon=\epsilon (d, \Gamma)$.
\end{proposition}

Fix $0<\delta_1<\ldots <\delta_n$ arbitrarily small and for $\fV\in \cV_n$ let us introduce  
\begin{equation}
\Lambda (\fV,\rho_n )\coloneqq
\{\xi \in \Omega_\tau\colon |\langle \nabla _\xi A^0(\xi),\theta\rangle|\le \rho_n|\theta| \ \  \forall \theta\in \cV\}
\label{eqn-2.15}
\end{equation}
with  $\rho_n=\varepsilon^{\frac{1}{2}}h^{-\delta_n}$. 

We define $\Xi_n$ by induction. First, $\Xi_d=\emptyset$. Assume that we defined $\Xi_{d},\ldots, \Xi_{n+1}$. Then we define 
\begin{equation}
\Xi_n \coloneqq \bigcup_{\fV\in \cV_n,\ \xi \in \Lambda(\fV)\cap \Omega_\tau} (\xi +\fV)\cap\Omega_\tau.
\label{eqn-2.16}
\end{equation}

It follows from Proposition~\ref{prop-2.5} that 
\begin{proposition}\label{prop-2.6}
Let $\varkappa>0$ in the definition of $\Theta'$ and $\delta>0$ in the definition of $\Omega_\tau$ be sufficiently small\,\footnote{\label{foot-11} They depend on $\vartheta$ and $\delta_1,\ldots,\delta_n$.}. Then for sufficiently small $h$

\begin{enumerate}[label=(\roman*), wide, labelindent=0pt]
\item\label{prop-2.6-i}
$\Xi_n \subset \bigcup_{\fV\in \cV} \Lambda (\fV, 2\rho_n)$.

\item\label{prop-2.6-ii}
If $\xi\notin \Xi_{n+1}$ and $\xi \in \xi'+\fV$, $\xi \in \xi''+\fW$ for
$\xi' \in \Lambda (\fV)$,  $\xi''\in \Lambda (\fW)$ with $\fV,\fW\in \cV_n$, then $\fV=\fW$.
\end{enumerate}
\end{proposition}

\begin{corollary}\label{cor-2.7}
Let $\varkappa>0$ in the definition of $\Theta'$ and $\delta>0$ in the definition of $\Omega_\tau$ be sufficiently small\,\footref{foot-11}.
Let $h$ be sufficiently small.

Then for each $\xi\in \Xi_n\setminus\Xi_{n+1}$ is defined just one $\fV=\fV(\xi)$ such that 
$\xi \in \xi'+\fV$ for some $\xi'\in \Lambda (\fV)$.
\end{corollary}

\vglue-10pt
\begin{claim} \label{eqn-2.17}
We slightly change definition of $\Xi_n$: $\xi =h(\gamma+\upxi)\in \Xi_{n,\new}$ iff $h\gamma\in \Xi_{n}$. 
From now on $\Xi_n\coloneqq \Xi_{n,\new}$.
\end{claim}

Consider $\xi',\xi''\in \Xi_n \setminus \Xi_{n+1}$. We say that $\xi'\cong \xi''$ if there exists $\xi\in \fV$, $\fV\in \cV$ such that $\xi',\xi''\in \xi +\fV$ and if $\xi'-\xi''\in \Gamma$. In virtue of above

\begin{claim} \label{eqn-2.18}
This relation is reflexive, symmetric and transitive.
\end{claim}

For $\xi\in \Xi_n$ we define 
\begin{gather}
\fX(\xi)=\{\xi'\colon \xi'\cong \xi\}.
\label{eqn-2.19}\\
\shortintertext{Then}
\diam (\fX(\xi)) \le C\rho_{d-1}.
\label{eqn-2.20}
\end{gather}

\section{Structure of operator $\cA$}
\label{sect-2.3}

For  $\xi\in \Xi_n \setminus \Xi_{n+1}$ denote by $\fH(\xi)$ the subspace $\sL^2(\cO)$ consisting of functions of the form
\begin{equation}
\sum_{\xi' \in\fX( \xi) }  c_{\xi'} e^{i\langle x ,\xi'\rangle}.
\label{eqn-2.21}
\end{equation}

In virtue of the properties of $\cA$ and $\cB$ and of resonant sets we arrive to

\begin{proposition}\label{prop-2.8}
Let $\varkappa>0$ in the definition of $\Theta'$ and $\delta>0$ in the definition of $\Omega_\tau$ be sufficiently small\,\footref{foot-11}.
Let $h$ be sufficiently small. 

Then for $\xi \in  \Xi _n\setminus \Xi_{n+1}$ operators $\cB$ and $\cA$ transform  $\fH(\xi)$  into $\fH(\xi)$.
\end{proposition}

Let us denote by $\cA_\gamma(\upxi)$ and $\cB_\gamma(\upxi)$ restrictions of $\cA$ and $\cB$ to $\fH(h(\gamma+\upxi))$. Here for $n=0$ we consider $\Xi_0$ to be the set of all non-resonant points and $\fX(\xi)=\{\xi\}$ for $\xi\in \Xi_0$. 

Then due to Propositions~\ref{prop-2.5} and~\ref{prop-2.8} we arrive to

\begin{proposition}\label{prop-2.9}
\begin{enumerate}[label=(\roman*), wide, labelindent=0pt]
\item\label{prop-2.9-i}
For each point $\lambda \in \Spec (A(\upxi))\cap \{|\lambda-\tau|\le  \varepsilon h^{-\delta}\}$ exists 
$\gamma\in  \Gamma^*$ such that $\xi= h(\gamma+\upxi)\in \Omega_\tau$ and 
$\dist (\lambda , \Spec (\cA_\gamma(\upxi))\le Ch^M$.
 
\item\label{prop-2.9-ii}
Conversely, for each point $\lambda \in \Spec (\cA_\gamma (\upxi))\cap \{|\lambda-\tau|\le  
\varepsilon h^{-\delta}\}$ \\ 
and $\xi= h(\xi+\gamma)$, $\dist (\lambda , \Spec (A(\upxi))\le Ch^M$.

\item\label{prop-2.9-iii}
Further, if $\lambda \in \Spec (A_\gamma (\upxi))\cap \{|\lambda-\tau|\le  \varepsilon h^{-\delta}\}$ is
a simple eigenvalue separated from the rest of  $\Spec (A(\upxi))$ by a distance at least $2h^{M-1}$, then
there exist   $\gamma$ and $\lambda'$, such that for $\xi=h(\gamma+\upxi)$,  $\lambda '\in \Spec (\cA(\xi))\cap \{|\lambda'-\lambda |\le  Ch^M\}$,
separated from the rest of $\Spec (\cA_\gamma(\xi))$ by a distance at least $h^{M-1}$ and from 
 $\bigcup_{\gamma'\in \Gamma^*,\ \gamma'\ne \gamma} \Spec (\cA_{\gamma'}(\upxi)$ by a distance at least $h^{M-1}$ as well.

\item\label{prop-2.9-iv}
Conversely, if  $\lambda '\in \Spec (\cA_\gamma(\upxi))\cap \{|\lambda-\tau|\le  \varepsilon h^{-\delta}\}$ is
a simple eigenvalue separated from the rest of $\Spec (\cA_\gamma(\upxi))$  by a distance at least $2h^{M-1}$, and also separated from  $\bigcup_{\gamma'\in \Gamma^*,\ \gamma'\ne \gamma} \Spec (\cA_{\gamma'}(\upxi))$ by a distance at least $2h^{M-1}$, 
then there exists   $\lambda \in \Spec (A(\upxi))\cap \{|\lambda'-\lambda |\le  Ch^M\}$
separated from the rest of $\Spec (A(\upxi))$  by a distance at least $h^{M-1}$.
\end{enumerate}
\end{proposition}

\begin{proof}
Proof is trivial.
\end{proof}

\chapter{Proof of Theorem~\ref{thm-1.3}}
\label{sect-3}
\section{Choosing \texorpdfstring{$\gamma^*$}{\textgamma$^*$}}
\label{sect-3.1}

The first approximation is $\xi^* \in \Sigma_\tau $ satisfying  (\ref{eqn-1.14}). Any  $\xi\in \Sigma_\tau$ in  $\epsilon'$-vicinity of $\xi^*$ also fits provided $\epsilon'>0$ is sufficiently small.

\begin{claim}\label{eqn-3.1}
One can select $\xi^*_\new \in \Sigma_\tau$  such that $|\xi^*_\new -\xi^*|\le h^\delta$ and  $\xi^*_\new$ satisfies \textup{(\ref{eqn-1.14})} with $\rho=\gamma \coloneqq h^\delta$. Here $\delta>0$ is arbitrarily small and $\varkappa = \varkappa(\delta)$.
\end{claim}

Indeed, it  follows from (\ref{eqn-2.5}). From now on $\xi^*\coloneqq \xi^*_\new$.

Then, according to Proposition~\ref{prop-2.1} we can diagonalize operator in \mbox{$\gamma$-vicinity} of $\xi^*$ and there $\rho=\gamma$. Then there
\begin{gather}
|\nabla^\alpha (\cA^0-A^0)|\le C_\alpha (\varepsilon +\varepsilon^2 \rho^{-2-|\alpha|})
\label{eqn-3.2}\\
\shortintertext{and in particular}
|\nabla^\alpha (\cA^0-A^0)|\le C h^\delta \qquad \text{for\ \ }|\alpha|\le 2.
\label{eqn-3.3}\\
\shortintertext{Let}
\Sigma'_\tau = \{\xi\colon \cA^0(\xi) =\tau\}.
\label{eqn-3.4}
\end{gather}

Then in the non-resonant points we are interested in functions $\lambda_\gamma(\upxi)=\cA^0(h(\gamma +\upxi))$ rather than in 
$\lambda^0_\gamma(\upxi)=A^0(h(\gamma +\upxi))$. One can prove easily the following statements:

\begin{proposition}\label{prop-3.1}
\begin{enumerate}[label=(\roman*), wide, labelindent=0pt]
\item\label{prop-3.1-i}
One can select $\xi^*\coloneqq\xi^*_\new\in \Sigma'_\tau$ satisfying \textup{(\ref{eqn-1.14})} and non-resonant with $\rho=\gamma$.
\item\label{prop-3.1-ii}
Further, all antipodal to $\xi^*$ points $\xi^*_1$,\ldots, $\xi^*_{2p-1}$ have the same properties.
\end{enumerate}
\end{proposition}

Let $\xi^*\eqqcolon h(\gamma^* +\upxi^*)$, $\gamma^*\in \Gamma^*$ and $\upxi^*\in \cO^*$. Then 
values in the nearby points are sufficiently separated:
\begin{equation}
|\lambda_\gamma (\upxi)-\lambda_{\gamma^*}(\upxi)|\ge \epsilon h^{1+\delta}
\qquad \forall \gamma\colon |\gamma-\gamma^*|\le Kh^{-\varkappa}\ \ \forall \upxi\in \cO^*.
\label{eqn-3.5}
\end{equation}
Indeed, $|\gamma-\gamma^*|\le Kh^{-\varkappa}$ implies that $(\gamma-\gamma^*)\in \Theta'_K$ and then
\begin{gather*}
|\langle \nabla \cA^0(\xi^*),\gamma-\gamma^*\rangle|\ge \gamma\\
\shortintertext{while} 
|\lambda_\gamma (\upxi)-\lambda_{\gamma^*}(\upxi)-h\langle \nabla \cA^0(\xi^*),\gamma-\gamma^*\rangle|
\le Ch^{3-3\varkappa}.
\end{gather*}

\section{Non-resonant points}
\label{sect-3.2}

Consider other non-resonant points (with $\rho =\varepsilon^{1/2}h^{-\delta}$). Let us determine how $\lambda_\gamma(\upxi)$ changes when we change $\upxi$. Due to (\ref{eqn-3.3}) 
\begin{equation}
\updelta \lambda_\gamma \coloneqq \lambda_\gamma (\upxi + \updelta\upxi)-\lambda_\gamma(\upxi)=
h \langle \nabla A^0(\xi), \updelta\upxi\rangle + O(h^2|\updelta\upxi|^2).
\end{equation}

To preserve  $\lambda _{\gamma^*} (\upxi)=\tau$ in the \emph{linearized settings\/} we need to shift $\upxi$ by $\updelta \upxi$ which is orthogonal to $\nabla _\xi \cA (\xi^*)$.

Let us take $\updelta\upxi=t\eta$
\begin{equation}
\ell\colon |\eta|=1, \qquad \langle \nabla \cA^0(\xi^*),\eta\rangle =0.
\label{eqn-3.7}
\end{equation}
Then in all non-resonant $\xi$ the shift will be $\langle \nabla _\xi \cA (\xi),\updelta\upxi\rangle$ with an absolute value 
$|\langle \nabla _\xi \cA (\xi),\eta\rangle|\cdot |t|$.

\paragraph{Case $d=2$.}
Let us start from the easiest case $d=2$. Without any loss of  the generality we assume that $\upxi^*$ is strictly inside $\cO^*$ (at the distance at least $C\epsilon^*$ from the border).  Then there is just one tangent direction $\eta$ and 
\begin{equation}
|\langle \nabla _\xi \cA^0 (\xi)|_{\xi =h\gamma},\eta\rangle |\asymp |\sin \varphi (\gamma^*,\gamma)|,
\asymp  h\min_{1\le k\le 2p} |\gamma-\gamma^*_k| 
\label{eqn-3.8}
\end{equation}
where $\varphi (\gamma^*,\gamma)$ is an angle between  $\nabla _\xi \cA^0(\xi)|_{\xi= h\gamma^*}$ and 
$\nabla _\xi \cA^0(\xi)|_{\xi= h\gamma}$, and $\xi^*_1,\ldots ,\xi^*_{2p-1}$ are antipodal points, and $\xi^*_{2p}=\xi^*$.

As long as $\min_{1\le k\le 2p} |\gamma-\gamma^*_k| \gtrsim h^{1-\varkappa}$ we may replace here $\xi=h(\gamma+\upxi)$ by $\xi=h\gamma$ and $\cA^0$ by $\cA$. In the nonlinear settings to ensure that
\begin{equation}
\lambda_{\gamma^*} (\upxi^*+\updelta\upxi(t))=\tau
\label{eqn-3.9}
\end{equation}
we need to include in $\updelta \upxi(t)$ a correction: $\updelta \upxi(t)=t\eta + O(t^2)$ but still
\begin{equation}
 \frac{d\ }{dt} \lambda_{\gamma} (\upxi^*+\updelta\upxi(t)) \asymp
h\langle \nabla _\xi \cA (\xi)|_{\xi =h\gamma},\eta\rangle ^{-1}
\label{eqn-3.10}
\end{equation}
Then the set $\cT(\xi)\coloneqq \{t\colon |t|\le \epsilon_0,\, |\cA^0(\xi (t))-\tau| \le \upsilon h\}$ is an interval of the length 
$\asymp \upsilon$ and then the union of such sets over $\xi= h\gamma+\upxi$ with indicated $\gamma$ does not exceed 
$R\upsilon$ with
\begin{equation}
R\coloneqq \sum _\gamma |\langle \nabla _\xi \cA (\xi)|_{\xi =h\gamma},\eta\rangle|^{-1},
\label{eqn-3.11}
\end{equation}
where we sum over set $\{\gamma\colon \gamma-\gamma^*|\gtrsim h^{-\varkappa}\ \&\  |\lambda_\gamma (h\gamma)-\tau|\le 2Ch\}$. The last restriction 
is due to the fact that $\cT(\xi)\ne \emptyset$ only for points with $|\lambda_\gamma (h\gamma)-\tau|\le 2Ch$. 

One can see easily that $R\asymp h^{-1}|\log h|$. Then, as $R\upsilon \le \epsilon'$ the set 
$[-\epsilon_0,\epsilon_0] \setminus \bigcup_{\gamma} \cT(h(\gamma +\upxi))$ contains an interval of the length $\ell = \upsilon $ and for all $t$, belonging to this interval, 
\begin{equation}
|\lambda_\gamma (h(\gamma +\upxi +\updelta \upxi(t)))-\tau |\ge \epsilon \upsilon h.
\label{eqn-3.12}
\end{equation} 
Then we need to take $\upsilon = \epsilon R^{-1}= \epsilon h|\log h|^{-1}$ and for $d=2$ as far as non-resonant are concerned, Theorem~\ref{thm-1.3} is almost proven\footnote{\label{foot-12} We need to cover almost antipodal points  and it will be done in the end of this subsection. We need to consider resonant points  and as well, and it will be done in the next subsection.}.

\paragraph{Case $d\ge 3$.} In this case we need to be more subtle and to make $(d-1)$ steps. We start from the point
$\xi^*=h(\gamma^*+\upxi^*)$; again without any loss of  the generality we assume that $\upxi^*$ is strictly inside $\cO^*$ (at the distance at least $C\epsilon^*$ from the border). Then after each step below it still will be the case (with decreasing constant).

\begin{enumerate}[label={\emph{Step~\Roman*}}, wide, labelindent=0pt]

\item\label{3.2-step-I}\hskip-\labelsep.\ \
On the first step we select  $\eta =\eta_1$ and consider only $\gamma $ such that (\ref{eqn-3.8}) holds; more precisely, the left-hand expression needs to be greater than the right-hand expression, multiplied by $\epsilon$\,\footnote{\label{foot-13} One can see easily, that the opposite holds.}. Then $R\asymp h^{1-d}$ and therefore exists $\upxi^*$ such that 
$\lambda_{\gamma^*} (\upxi^*)=\tau$ and 
$|\lambda_{\gamma} (\upxi^*)-\tau|\ge \epsilon \upsilon_1 h$ with 
$\upsilon_1=\epsilon h^{d-1}$ for all $\gamma$ indicated above.

\item\label{3.2-step-II}\hskip-\labelsep.\ \
On the second step we select $\eta =\eta_2$  perpendicular to $\eta_1$. To preserve inequality (\ref{eqn-3.12}) (with smaller constant $\epsilon)$ for $\gamma$, already covered by \ref{3.2-step-I}, we need to take $|\updelta \xi|\le \epsilon' \upsilon_1$ and consider $\updelta\upxi = t\eta_2 +O(t^2)$.

Then the same arguments as before results in inequality (\ref{eqn-3.12}) with
$\upsilon \coloneqq \upsilon_2= \epsilon R^{-1}\upsilon_1$ for a new bunch of points. Then for $d=3$ as far as non-resonant are concerned, Theorem~\ref{thm-1.3} is almost proven\footref{foot-13}.

\item[\emph{Next steps}]\label{3.2-step-III}\hskip-\labelsep.\ \
Continuing this process, on $k$-th step we select $\eta_k$ orthogonal to $\eta_1,\ldots,\eta_{k-1}$. Then we get $\upsilon_k =\epsilon R^{-1}\upsilon_{k-1}$ and on the last 
 $(d-1)$-th step we achieve a separation at least $\upsilon_{d-1} = \epsilon R^{1-d}$.
\end{enumerate} 

\begin{remark}
In Subsection~\ref{sect-4.1} we discuss how to increase $\upsilon$ for $d\ge 3$.
\end{remark}

\paragraph{Almost antipodal points.}
We need to cover points with $|\xi-\xi^*_k|\le h^{1-\kappa}$ for $k=1,\ldots, 2p-1$ and as we already know for each $k$ (and fixed $\upxi)$ there exists no more than one such point $\xi=h(\gamma +\upxi)$ with $|\lambda_\gamma (\upxi)-\tau|\lesssim h^{1+\delta}$. 

We take care of such points during \ref{3.2-step-I}. Observe that during this step we automatically take care of any point with
\begin{equation}
|\nabla _\xi \cA^0(\xi) ,\eta_1\rangle| \ge \epsilon h,
\label{eqn-3.13}
\end{equation}
assuming that $|t|\le \epsilon_0$ with sufficiently small $\epsilon_0=\epsilon_0(\epsilon)$.

Let us select $\eta_1$ so that on $\eta_1$  quadratic forms at points $\xi^*_1,\ldots, \xi^*_{2p-1}$ in condition (\ref{eqn-1.14}) are different from one at point $\xi^*$ by at least $\epsilon_0$. Then for each $j=1,\ldots, 2p-1$ the the measure of the set 
\begin{equation*}
\{t\colon |t|\le \epsilon_0, \, |\lambda_{\gamma_j} (\upxi +\updelta \upxi(t))|\le \upsilon h\}
\end{equation*}
does not exceed $C h^{-1}(\upsilon h)^{\frac{1}{2}}$, and then the measure of the union of such sets (by $j$) also does not exceed and therefore for $\upsilon_1 = \epsilon_1 h^{d-1}$ (for $d\ge 3$) and $\upsilon_1 = \epsilon_1 h |\log h|^{-1}$ (for $d=2$) with sufficiently small $\epsilon_1$  we can find $t\colon |t|\le \epsilon_0$ so that condition (\ref{eqn-3.8}) is fulfilled for all non-resonant points.

\section{Resonant points}
\label{sect-3.3}

Next on this step we need to separate $\lambda _{\gamma^*}(\upxi)$ from all $\lambda_n (\upxi)$ (save one, coinciding with it) by the distance at least $\upsilon h $ by choosing $\upxi$. We can during the same steps as described in the previous section: let $\lambda _{\gamma ,j }(\upxi)$ denote eigenvalues of $\cA_\gamma (\upxi)$ with $j=\#\fX(\gamma h)$. 

Observe that both $\cA_\gamma (\upxi)$ and $\#\fX(\gamma h)$ depend on the equivalency class $[\gamma]$ of $\gamma$ rather than on $\gamma$ itself and that 
\begin{equation}
\sum_{[\gamma]} \#\fX(\gamma h)=\sum_{1\le n \le d-1}\#(\Xi_n) =O(h^{1-d+\sigma'}+ \varepsilon^{3/2} h^{-d-\sigma}),
\label{eqn-3.14}
\end{equation}
where on the left $[\gamma]$ runs over all equivalency classes with  $\gamma\in \bigcup_{1\le n \le d-1}\Xi_n$. 

We also observe that for resonant points 
\begin{equation}
|\sin \varphi (\xi,\xi^*)| \ge \epsilon h^\delta
\label{eqn-3.15}
\end{equation}
and therefore for $\lambda'_{\gamma}$, which are eigenvalues of $\cA^0(h(\gamma+\upxi))$\,\footnote{\label{foot-14} Recall, that $\cA^0$ is diagonal matrix.} (\ref{eqn-3.10}) holds and signs are the same for $\gamma$ in the same block. On the other hand,
\begin{equation}
| \frac{d\ }{dt} \cB  (h(\gamma+\upxi^*+\updelta\upxi(t))| \le C \varepsilon h \ll h^{1+\delta'}
\label{eqn-3.16}
\end{equation}
and therefore for $\lambda_{\gamma, j}(t)$ which are eigenvalues of $\cA_{[\gamma]}(\upxi)$ (\ref{eqn-3.10}) sill holds.

Therefore the  arguments of each \ref{3.2-step-I}, \ref{3.2-step-II}  etc extends to resonant points as well. However the number of \emph{new points} to be taken into account on each step is given by the right-hand expression of (\ref{eqn-3.14}) and therefore $R$ needs to be redefined
\begin{equation}
R \coloneqq h^{1-d}+ \varepsilon^{3/2} h^{-d-\sigma}.
\label{eqn-3.17}
\end{equation}

This leads to the final expression (\ref{eqn-1.15}) for $\upsilon$. Theorem~\ref{thm-1.3} is proven.

\chapter{Discussion}
\label{sect-4}

\section{Improving $\upsilon$}
\label{sect-4.1}

Can we improve (increase) expression for $\upsilon$ given by (\ref{eqn-1.15})? I do not know if one can do anything  with the restriction $\upsilon \le\epsilon \varepsilon ^{-\frac{3}{2}(d-1)}h^{d(d-1)}$ which is due to resonant points, but restriction $\upsilon \le \epsilon  h^{(d-1)^2}$ could be improved for $d\ge 3$, which makes sense only if
\begin{equation}
h\le h \le h^{\frac{2}{3}-\sigma}.
\label{eqn-4.1}
\end{equation}

Indeed, on   Step $n$, $n\ge 2$,  we need to take into account only non-resonant points belonging to the set
\begin{equation}
\cJ \coloneqq  h(\Gamma^* +\upxi) \cap \{\xi\colon |\cA^0 (\xi)-\tau|\le C\upsilon_{n-1}h\}.
\label{eqn-4.2}
\end{equation}
Determination of upper estimate for such number falls into realm of the Number Theory. I am familiar only with the estimate
\begin{equation}
\# \cJ  \le Ch^{1-d} \upsilon_{n-1} + Ch^{2-d-2/(d+1)},
\label{eqn-4.3}
\end{equation}
which follows from Theorem at page 224 of \cite{Gui}. Probably it was improved, but those improvement have no value here. 

The second term in the right-hand expression of (\ref{eqn-4.2}) is larger, however, the second term  in the right-hand expression of (\ref{eqn-3.17}) is still larger and therefore on each Step $n\ge 2$ we have $R \coloneqq \varepsilon^{3/2} h^{-d-\sigma}$, and this leads us to the following improvement of Theorem~\ref{thm-1.3}:

\begin{theorem}\label{thm-4.1}
In the framework of Theorem~\ref{thm-1.1}, under additional assumptions $d\ge 3$ and \textup{(\ref{eqn-4.1})}, the statement of Theorem~\ref{thm-1.3} holds with 
\begin{equation}
\upsilon=\epsilon \varepsilon^{-3(d-2)/2} h^{d^2-d-1-\sigma} 
\label{eqn-4.4}
\end{equation}
with arbitrarily small exponent $\sigma>0$. In particular, $\upsilon= \epsilon h^{d^2-5d/2+3-\sigma}$ for $\varepsilon=h$.
\end{theorem}

\begin{remark}\label{rem-4.2}
One can try to improve further (\ref{eqn-4.4}) for $\varepsilon \le h$. In this case resonant points become the main obstacle. In the definition of resonant points we need to take $\rho=h^{1/2-\delta}$; however only resonant points $\xi$ with
\begin{equation*}
\fX(\xi) \cap \{ \xi' \colon |\cA^0(\xi') -\tau| \le C(\upsilon_{n-1}h+\varepsilon)\}\ne \emptyset
\end{equation*}
should be taken into account on $n$-th step.
\end{remark}

\section{Condition (\ref{eqn-1.14})}
\label{sect-4.2}
We know  that for connected component $\Sigma_\tau$ this condition (\ref{eqn-1.14}) is fulfilled automatically. 

On the other hand, let $\Sigma_\tau=\bigcup_{1\le j\le p}\Sigma_\tau ^{(j)}$ with connected $\Sigma_\tau ^{(j)}$. Let $p=2$ and  condition (\ref{eqn-1.14})  be violated at each point of $\Sigma_\tau$. Does it mean that $\Sigma_\tau ^{(2)}=\Sigma_\tau ^{(1)}+ \eta$ (i.e. $\Sigma_\tau ^{(1)}$ shifted by $\eta$)?

Next, let $\Sigma_\tau ^{(j)}=\Sigma_\tau ^{(1)}+ \eta_j$ for $j=2,\ldots, q$. Then for these components instead of condition (\ref{eqn-1.14}) one can impose the similar condition involving level surfaces $\sigma_{\tau,\varepsilon}$ of $\cA^0(\xi) +\varepsilon B_0(\xi)$. 
This would affect only \ref{3.2-step-I} of our analysis, leading to $\upsilon_{1,\new}\coloneqq \min(\varepsilon h, \upsilon_1)$ with $\upsilon_1$ defined without taking into account of antipodal points. Since $\upsilon_1\le h^{d-1}$ anyway, under assumption $\varepsilon \ge h$ 
we get the same formulae for $\upsilon$ for $d\ge 3$ as stated in Theorems~\ref{thm-1.3} and~\ref{thm-4.1}, while for $d=2$ we get
\begin{equation}
\upsilon = h\min (\varepsilon , \, \varepsilon^{-3/2}h^\sigma ).
\label{eqn-4.5}
\end{equation}

\section{Differentiability}
\label{sect-4.3}
Definitely  our result would follow from the asymptotics of the \emph{density of states}
\begin{gather}
\N'_h(\tau)\coloneqq \frac{d \N_h(\tau)}{d\tau}= (\kappa'_0(\tau) +o(1))h^{-d} \qquad \text{as\ \ }h\to +0,
\label{eqn-4.6}\\
\shortintertext{where}
\N_h(\tau) = \int_{\cO^*} \N_h(\upxi,\tau)\,d\upxi 
\label{eqn-4.7}\\
\intertext{is an \emph{integrated  density of states:}}
\N_h(\upxi,\tau)=\# \{\mu <\tau,\, \mu \text{\ \ is an eigenvalue of\ \ } A_h(\upxi)\} .
\label{eqn-4.8}
\end{gather}
However, despite $\N_h(\upxi,\tau)$ has a complete asymptotics (see f.e. \cite{IvrIDS}), we do not know anything about asymptotics of $\N'_h(\tau)$ (even $\N'_h(\tau)\asymp h^{-d}$ is unknown).

\end{document}